\documentclass[12pt]{article}
\usepackage{amsfonts}
\usepackage{amsmath,amscd,amsthm,a4,amssymb}
\usepackage{amsmath,amscd,amsthm,a4,amssymb}

\newtheorem{theorem}{Theorem}
{}
\newtheorem{corollary}{Corollary}
{}

{}
\newtheorem{remark}{Remark}
{}

{}
\theoremstyle{plain}
{}

{}

{}

{}

{}

{}

{}

{}

{}

{}

{}

{}

{}

{}
\newtheorem{proposition}{Proposition}
{}

{}

\begin{document}
\begin{center}
{\Large \bf{Flat Rotational Surface with Pointwise 1-type Gauss map in E$^{4}$ \ }}
\end{center}
\centerline{\large Ferdag KAHRAMAN AKSOYAK  $^{1}$, Yusuf YAYLI $^{2}${\footnotetext{
{ E-mail: $^{1}$ferda@erciyes.edu.tr(F. Kahraman Aksoyak ); $^{2}$yayli@science.ankara.edu.tr (Y.Yayli)}} }}

\
\centerline{\it $^{1}$Erciyes University, Department of Mathematics,
Kayseri, Turkey}

\centerline{\it $^{2}$Ankara University, Department of Mathematics,
Ankara, Turkey}

\begin{abstract}
In this paper we study general rotational surfaces in the 4- dimensional
Euclidean space $\mathbb{E}^{4}$ and give a characterization of flat general
rotation surface with pointwise 1-type Gauss map. Also, we show that a
non-planar flat general rotation surface with pointwise 1-type Gauss map is
a Lie group if and only if it is a Clifford torus.
\end{abstract}

\begin{quote}\small
\end{quote}
\begin{quote}\small
\end{quote}

\section{Introduction}

A submanifold $M$ of a Euclidean space $\mathbb{E}^{m}$ is said to be of
finite type if its position vector $x$ can be expressed as a finite sum of
eigenvectors of the Laplacian $\Delta $ of $M$, that is, $%
x=x_{0}+x_{1}+...x_{k}$, where $x_{0}$ is a constant map, $x_{1},...,x_{k}$
are non-constant maps such that $\Delta x_{i}=\lambda _{i}x_{i},$ $\lambda
_{i}\in $ $\mathbb{R}$, $i=1,2,...,k.$ If $\lambda _{1},\lambda _{2},$...,$%
\lambda _{k}$ are all different, then $M$ is said to be of $k-$type. This
definition was similarly extended to differentiable maps, in particular, to
Gauss maps of submanifolds \cite{chen1}.

If a submanifold $M$ of a Euclidean space or pseudo-Euclidean space has
1-type Gauss map $G$, then $G$ satisfies $\Delta G=\lambda \left( G+C\right)
$ for some $\lambda \in \mathbb{R}$ and some constant vector $C.$ Chen and
Piccinni made a general study on compact submanifolds of Euclidean spaces
with finite type Gauss map and they proved that a compact hypersurface $M$
of $\mathbb{E}^{n+1}$ has 1-type Gauss map if and only if $M$ is a
hypersphere in $\mathbb{E}^{n+1}$ \cite{chen1}$.$

Hovewer the Laplacian of the Gauss map of some typical well known surfaces
such as a helicoid, a catenoid and a right cone in Euclidean 3-space $%
\mathbb{E}^{3}$ take a somewhat different form, namely,
\begin{equation}
\Delta G=f\left( G+C\right)
\end{equation}
for some smooth function $f$ on $M$ and some constant vector $C.$ A
submanifold $M$ of a Euclidean space $\mathbb{E}^{m}$ is said to have
pointwise 1-type Gauss map if its Gauss map satisfies $\left( 1\right) $ for
some smooth function $f$ on $M$ and some constant vector $C.$ A submanifold
with pointwise 1-type Gauss map is said to be of the first kind if the
vector $C$ in $\left( 1\right) $ is zero\ vector. Otherwise, the pointwise
1-type Gauss map is said to be of the second kind.

Surfaces in Euclidean space and in pseudo-Euclidean space with pointwise
1-type Gauss map were recently studied in \cite{choi1}, \cite{choi2}, \cite%
{dursun2}, \cite{dursun3}, \cite{dursun4}, \cite{kim2}, \cite{kim1}. Also
Dursun and Turgay in \cite{dursun1} gave all general rotational surfaces in $%
\mathbb{E}^{4}$ with proper pointwise 1-type Gauss map of the first kind and
classified minimal rotational surfaces with proper pointwise 1-type Gauss
map of the second kind. Arslan et al. in \cite{arslan1} investigated
rotational embedded surface with pointwise 1-type Gauss map. Arslan at el.
in \cite{arslan2} gave necessary and sufficent conditions for Vranceanu
rotation surface to have pointwise 1-type Gauss map. Yoon in \cite{yoon2}
showed that flat Vranceanu rotation surface with pointwise 1-type Gauss map
is a Clifford torus.

In this paper, we study general rotational surfaces in the 4- dimensional
Euclidean space $\mathbb{E}^{4}$ and give a characterization of flat general
rotation with pointwise 1-type Gauss map. Also, we show that a non-planar
flat general rotation surface with pointwise 1-type Gauss map is a Lie group
if and only if it is a Clifford torus.\

\section{Preliminaries}

Let $M$ be an oriented $n-$dimensional submanifold in $m-$dimensional
Euclidean space $\mathbb{E}^{m}.$ Let $e_{1},$...,$e_{n},e_{n+1},$...,$e_{m}$
be an oriented local orthonormal frame in $\mathbb{E}^{m}$ such that $e_{1},$%
...,$e_{n}$ are tangent to $M$\ and $e_{n+1},$...,$e_{m}$ normal to $M.$ We
use the following convention on the ranges of indices: $1\leq i,j,k,$...$%
\leq n$, $n+1\leq r,s,t,$...$\leq m$, $1\leq A,B,C,$...$\leq m.$

Let $\tilde{\nabla}$ be the Levi-Civita connection of $\mathbb{E}^{m}$ and $%
\nabla $ the induced connection on $M$. Let $\omega _{A}$ be the dual-1 form
of $e_{A}$ defined by $\omega _{A}\left( e_{B}\right) =\delta _{AB}$. Also,
the connection forms $\omega _{AB}$ are defined by%
\begin{equation*}
de_{A}=\sum \limits_{B}\omega _{AB}e_{B},\text{ \  \ }\omega _{AB}+\omega
_{BA}=0.
\end{equation*}%
Then we have
\begin{equation}
\tilde{\nabla}_{e_{k}}^{e_{i}}=\sum \limits_{j=1}^{n}\omega _{ij}\left(
e_{k}\right) e_{j}+\sum \limits_{r=n+1}^{m}h_{ik}^{r}e_{r}
\end{equation}%
and%
\begin{equation}
\tilde{\nabla}_{e_{k}}^{e_{s}}=-A_{s}(e_{k})+\sum \limits_{r=n+1}^{m}\omega
_{sr}\left( e_{k}\right) e_{r},\text{ \ }D_{e_{k}}^{e_{s}}=\sum%
\limits_{r=n+1}^{m}\omega _{sr}\left( e_{k}\right) e_{r},
\end{equation}%
where $D$ is the normal connection, $h_{ik}^{r}$ the coefficients of the
second fundamental form $h$ and $A_{s}$ the Weingarten map in the direction $%
e_{s}.$

For any real function $f$ on $M$ the Laplacian of $f$ is defined by
\begin{equation}
\Delta f=-\sum \limits_{i}\left( \tilde{\nabla}_{e_{i}}\tilde{\nabla}%
_{e_{i}}f-\tilde{\nabla}_{\nabla _{e_{i}}^{e_{i}}}f\right) .
\end{equation}

If we define a covariant differention $\tilde{\nabla}h$ of the second
fundamental form $h$ on the direct sum of the tangent bundle and the normal
bundle $TM\oplus T^{\perp }M$ of $M$ by
\begin{equation*}
\left( \tilde{\nabla}_{X}h\right) \left( Y,Z\right) =D_{X}h\left( Y,Z\right)
-h\left( \nabla _{X}Y,Z\right) -h\left( Y,\nabla _{X}Z\right)
\end{equation*}%
for any vector fields $X,$ $Y$ and $Z$ tangent to $M.$ Then we have the
Codazzi equation%
\begin{equation}
\left( \tilde{\nabla}_{X}h\right) \left( Y,Z\right) =\left( \tilde{\nabla}%
_{Y}h\right) \left( X,Z\right)
\end{equation}%
and the Gauss equation is given by
\begin{equation}
\left \langle R(X,Y)Z,W\right \rangle =\left \langle h\left( X,W\right)
,h\left( Y,Z\right) \right \rangle -\left \langle h\left( X,Z\right)
,h\left( Y,W\right) \right \rangle ,
\end{equation}%
where the vectors $X,$ $Y,$ $Z$ and $W$ are tangent to $M$ and $R$ is the
curvature tensor associated with $\nabla $ and the curvature tensor $R$ is
defined by%
\begin{equation*}
R(X,Y)Z=\nabla _{X}\nabla _{Y}Z-\nabla _{Y}\nabla _{X}Z-\nabla _{\left[ X,Y%
\right] }Z.
\end{equation*}%
Let us now define the Gauss map $G$ of a submanifold $M$ into $G(n,m)$ in $%
\wedge ^{n}\mathbb{E}^{m},$ where $G(n,m)$ is the Grassmannian manifold
consisting of all oriented $n-$planes through the origin of $\mathbb{E}^{m}$
and $\wedge ^{n}\mathbb{E}^{m}$ is the vector space obtained by the exterior
product of $n$ vectors in $\mathbb{E}^{m}.$ In a natural way, we can
identify $\wedge ^{n}\mathbb{E}^{m}$ with some Euclidean space $\mathbb{E}%
^{N}$ where $N=\left(
\begin{array}{c}
m \\
n%
\end{array}%
\right) .$ The map $G:M\rightarrow G(n,m)\subset E^{N}$ defined by $%
G(p)=\left( e_{1}\wedge ...\wedge e_{n}\right) \left( p\right) $ is called
the Gauss map of $M,$ that is, a smooth map which carries a point $p$ in $M$
into the oriented $n-$plane through the origin of $\mathbb{E}^{m}$ obtained
from parallel translation of the tangent space of $M$ at $p$ in $.$

Bicomplex number is defined by the basis $\left \{ 1,i,j,ij\right \} $ where
$i,j,ij$ satisfy $i^{2}=-1,$ $j^{2}=-1,$ $ij=ji.$ Thus any bicomplex number $%
x $ can be expressed as $x=x_{1}1+x_{2}i+x_{3}j+x_{4}ij$, $\forall
x_{1},x_{2},x_{3},x_{4}\in \mathbb{R}.$ We denote the set of bicomplex
numbers by $C_{2}.$ For any $x=x_{1}1+x_{2}i+x_{3}j+x_{4}ij$ and $%
y=y_{1}1+y_{2}i+y_{3}j+y_{4}ij$ in $C_{2}$ the bicomplex number addition is
defined by
\begin{equation*}
x+y=\left( x_{1}+y_{1}\right) +\left( x_{2}+y_{2}\right) i+\left(
x_{3}+y_{3}\right) j+\left( x_{4}+y_{4}\right) ij\text{.}
\end{equation*}%
The multiplication of a bicomplex number $x=x_{1}1+x_{2}i+x_{3}j+x_{4}ij$ by
a real scalar $\lambda $ is given by
\begin{equation*}
\lambda x=\lambda x_{1}1+\lambda x_{2}i+\lambda x_{3}j+\lambda x_{4}ij\text{.%
}
\end{equation*}%
With this addition and scalar multiplication, $C_{2}$ is a real vector space.

Bicomplex number product, denoted by $\times $, over the set of bicomplex
numbers $C_{2}$\ is given by
\begin{eqnarray*}
x\times y &=&\left( x_{1}y_{1}-x_{2}y_{2}-x_{3}y_{3}+x_{4}y_{4}\right)
+\left( x_{1}y_{2}+x_{2}y_{1}-x_{3}y_{4}-x_{4}y_{3}\right) i \\
&&+\left( x_{1}y_{3}+x_{3}y_{1}-x_{2}y_{4}-x_{4}y_{2}\right) j+\left(
x_{1}y_{4}+x_{4}y_{1}+x_{2}y_{3}+x_{3}y_{2}\right) ij\text{.}
\end{eqnarray*}%
Vector space $C_{2}$ together with the bicomplex product $\times $ is a real
algebra.

Since the bicomplex algebra is associative, it can be considered in terms of
matrices. Consider the set of matrices%
\begin{equation*}
Q=\left \{ \left(
\begin{array}{cccc}
x_{1} & -x_{2} & -x_{3} & x_{4} \\
x_{2} & x_{1} & -x_{4} & -x_{3} \\
x_{3} & -x_{4} & x_{1} & -x_{2} \\
x_{4} & x_{3} & x_{2} & x_{1}%
\end{array}%
\right) ;\text{ \  \  \  \  \  \ }x_{i}\in \mathbb{R}\text{ ,\  \  \  \ }1\leq i\leq
4\right \} \text{.}
\end{equation*}%
The set $Q$ together with matrix addition and scalar matrix multiplication
is a real vector space. Furthermore, the vector space together with matrix
product is an algebra \cite{yay}.

The transformation
\begin{equation*}
g:C_{2}\rightarrow Q
\end{equation*}%
given by
\begin{equation*}
g\left( x=x_{1}1+x_{2}i+x_{3}j+x_{4}ij\right) =\left(
\begin{array}{cccc}
x_{1} & -x_{2} & -x_{3} & x_{4} \\
x_{2} & x_{1} & -x_{4} & -x_{3} \\
x_{3} & -x_{4} & x_{1} & -x_{2} \\
x_{4} & x_{3} & x_{2} & x_{1}%
\end{array}%
\right)
\end{equation*}%
is one to one and onto. Morever $\forall x,y\in C_{2}$ and $\lambda \in
\mathbb{R},$ we have
\begin{eqnarray*}
g\left( x+y\right) &=&g\left( x\right) +g\left( y\right) \\
g\left( \lambda x\right) &=&\lambda g\left( x\right) \\
g\left( xy\right) &=&g\left( x\right) g\left( y\right) \text{.}
\end{eqnarray*}%
Thus the algebras $C_{2}$ and $Q$ are isomorphic.

Let $x\in C_{2}.$ Then $x$ can be expressed as $x=\left( x_{1}+x_{2}i\right)
+\left( x_{3}+x_{4}i\right) j.$ In this case, there is three different
conjugations for bicomplex numbers as follows:%
\begin{eqnarray*}
x^{t_{1}} &=&\left[ \left( x_{1}+x_{2}i\right) +\left( x_{3}+x_{4}i\right) j%
\right] ^{t_{1}}=\left( x_{1}-x_{2}i\right) +\left( x_{3}-x_{4}i\right) j \\
x^{t_{2}} &=&\left[ \left( x_{1}+x_{2}i\right) +\left( x_{3}+x_{4}i\right) j%
\right] ^{t_{2}}=\left( x_{1}+x_{2}i\right) -\left( x_{3}+x_{4}i\right) j \\
x^{t_{3}} &=&\left[ \left( x_{1}+x_{2}i\right) +\left( x_{3}+x_{4}i\right) j%
\right] ^{t_{3}}=\left( x_{1}-x_{2}i\right) -\left( x_{3}-x_{4}i\right) j
\end{eqnarray*}

\section{Flat Rotation Surfaces with Pointwise 1-Type Gauss Map in $E^{4}$}

In this section, we consider the flat rotation surfaces with pointwise
1-type Gauss map in Euclidean 4- space. Let consider the equation of the
general rotation surface given in \cite{moore}.
\begin{equation*}
\varphi \left( t,s\right) =%
\begin{pmatrix}
\cos mt & -\sin mt & 0 & 0 \\
\sin mt & \cos mt & 0 & 0 \\
0 & 0 & \cos nt & -\sin nt \\
0 & 0 & \sin nt & \cos nt%
\end{pmatrix}%
\left(
\begin{array}{c}
\alpha _{1}(s) \\
\alpha _{2}(s) \\
\alpha _{3}(s) \\
\alpha _{4}(s)%
\end{array}%
\right) ,
\end{equation*}%
where $\alpha \left( s\right) =\left( \alpha _{1}\left( s\right) ,\alpha
_{2}\left( s\right) ,\alpha _{3}\left( s\right) ,\alpha _{4}\left( s\right)
\right) $ is a regular smooth curve in $\mathbb{E}^{4}$ on an open interval $%
I$ in $\mathbb{R}$ and $m$, $n$ are some real numbers which are the rates of
the rotation in fixed planes of the rotation. If we choose the meridian
curve $\alpha $ as $\alpha \left( s\right) =\left( x\left( s\right)
,0,y(s),0\right) $ is unit speed curve and the rates of the rotation $m$ and
$n\ $as $m=n=1,$ we obtain the surface as follows:
\begin{equation}
M:\text{ \ }X\left( s,t\right) =\left( x\left( s\right) \cos t,x\left(
s\right) \sin t,y(s)\cos t,y(s)\sin t\right)
\end{equation}%
Let $M$ be a general rotation surface in $\mathbb{E}^{4}$ given by $(7)$. We
consider the following orthonormal moving frame $\left \{
e_{1},e_{2},e_{3},e_{4}\right \} $ on $M$\ such that $e_{1},e_{2}$ are
tangent to $M$ and $e_{3},e_{4}$ are normal to $M:$
\begin{eqnarray*}
e_{1} &=&\frac{1}{\sqrt{x^{2}\left( s\right) +y^{2}(s)}}\left( -x\left(
s\right) \sin t,x\left( s\right) \cos t,-y(s)\sin t,y(s)\cos t\right) \\
e_{2} &=&\left( x^{\prime }\left( s\right) \cos t,x^{\prime }\left( s\right)
\sin t,y^{\prime }(s)\cos t,y^{\prime }(s)\sin t\right) \\
e_{3} &=&\left( -y^{\prime }(s)\cos t,-y^{\prime }(s)\sin t,x^{\prime
}\left( s\right) \cos t,x^{\prime }\left( s\right) \sin t\right) \\
e_{4} &=&\frac{1}{\sqrt{x^{2}\left( s\right) +y^{2}(s)}}\left( -y(s)\sin
t,y(s)\cos t,x\left( s\right) \sin t,-x\left( s\right) \cos t\right)
\end{eqnarray*}%
where $e_{1}=\frac{1}{\sqrt{x^{2}\left( s\right) +y^{2}(s)}}\frac{\partial }{%
\partial t}$ and $e_{2}=\frac{\partial }{\partial s}$. Then we have the dual
1-forms as:
\begin{equation*}
\omega _{1}=\sqrt{x^{2}\left( s\right) +y^{2}(s)}dt\text{ \  \  \  \ and \  \  \
\ }\omega _{2}=ds
\end{equation*}%
By a direct computation we have components of the second fundamental form
and the connection forms as:%
\begin{equation*}
h_{11}^{3}=b(s),\text{ \ }h_{12}^{3}=0,\text{ \ }h_{22}^{3}=c(s),
\end{equation*}%
\begin{equation*}
h_{11}^{4}=0,\text{ \ }h_{12}^{4}=-b(s),\text{ \ }h_{22}^{4}=0,
\end{equation*}%
\begin{eqnarray*}
\omega _{12} &=&-a(s)\omega _{1},\text{ \  \ }\omega _{13}=b(s)\omega _{1},%
\text{ \  \ }\omega _{14}=-b(s)\omega _{2} \\
\omega _{23} &=&c(s)\omega _{2},\text{ \  \ }\omega _{24}=-b(s)\omega _{1},%
\text{ \  \ }\omega _{34}=-a(s)\omega _{1}.
\end{eqnarray*}%
By covariant differentiation with respect to $e_{1}$ and $e_{2}$ a
straightforward calculation gives:
\begin{eqnarray}
\tilde{\nabla}_{e_{1}}e_{1} &=&-a(s)e_{2}+b(s)e_{3}, \\
\tilde{\nabla}_{e_{2}}e_{1} &=&-b(s)e_{4},  \notag \\
\tilde{\nabla}_{e_{1}}e_{2} &=&a(s)e_{1}-b(s)e_{4},  \notag \\
\tilde{\nabla}_{e_{2}}e_{2} &=&c(s)e_{3},  \notag \\
\tilde{\nabla}_{e_{1}}e_{3} &=&-b(s)e_{1}-a(s)e_{4},  \notag \\
\tilde{\nabla}_{e_{2}}e_{3} &=&-c(s)e_{2},  \notag \\
\tilde{\nabla}_{e_{1}}e_{4} &=&b(s)e_{2}+a(s)e_{3},  \notag \\
\tilde{\nabla}_{e_{2}}e_{4} &=&b(s)e_{1},  \notag
\end{eqnarray}%
where
\begin{equation}
a(s)=\frac{x(s)x^{\prime }(s)+y(s)y^{\prime }(s)}{x^{2}\left( s\right)
+y^{2}(s)},
\end{equation}%
\begin{equation}
\text{\ }b(s)=\frac{x(s)y^{\prime }(s)-x^{\prime }(s)y(s)}{x^{2}\left(
s\right) +y^{2}(s)},
\end{equation}%
\begin{equation}
c(s)=x^{\prime }(s)y^{\prime \prime }-x^{\prime \prime }y^{\prime }(s).
\end{equation}%
The Gaussian curvature is obtained by
\begin{equation}
K=\det \left( h_{ij}^{3}\right) +\det \left( h_{ij}^{4}\right)
=b(s)c(s)-b^{2}(s).
\end{equation}%
If the surface $M$ is flat, from $(12)$ we get
\begin{equation}
b(s)c(s)-b^{2}(s)=0.
\end{equation}%
Furthermore, by using the equations of Gauss and Codazzi after some
computation we obtain%
\begin{equation}
a^{\prime }\left( s\right) +a^{2}\left( s\right) =b^{2}(s)-b(s)c(s)
\end{equation}%
and
\begin{equation}
b^{\prime }\left( s\right) =-2a(s)b(s)+a(s)c(s),
\end{equation}%
respectively.

By using $\left( 4\right) $ and $\left( 8\right) $ and straight-forward
computation the Laplacian $\Delta G$ of the Gauss map $G$ can be expressed as%
\begin{eqnarray}
\Delta G &=&\left( 3b^{2}\left( s\right) +c^{2}\left( s\right) \right)
\left( e_{1}\wedge e_{2}\right) +\left( 2a(s)b(s)-a(s)c(s)-c^{\prime }\left(
s\right) \right) \left( e_{1}\wedge e_{3}\right)  \notag \\
&&+\left( -3a(s)b(s)-b^{\prime }(s)\right) \left( e_{2}\wedge e_{4}\right)
+\left( 2b^{2}(s)-2b(s)c(s)\right) \left( e_{3}\wedge e_{4}\right) .
\end{eqnarray}

\begin{remark}
\label{remark 1}Similar computations to above computations is given for
tensor product surfaces in \cite{arslan3} and for general rotational surface in \cite{dursun1}
\end{remark}

Now we investigate the flat rotation surface with the pointwise 1-type Gauss
map. From $(13)$, we obtain that $b(s)=0$ or $b(s)=c(s).$ We assume that $%
b(s)\neq c(s).$ Then $b(s)$ is equal to zero and $(15)$ implies that $%
a(s)c(s)=0.$ Since $b(s)\neq c(s),$ it implies that $c(s)$ is not equal to
zero. Then we obtain as $a(s)=0.$ In that case, by using $(9)$ and $(10)$ we
obtain that $\alpha \left( s\right) =\left( x\left( s\right)
,0,y(s),0\right) $ is a constant vector. This is a contradiction. Therefore $%
b(s)=c(s)$ for all $s.$ From $(14)$, we get
\begin{equation}
a^{\prime }\left( s\right) +a^{2}\left( s\right) =0
\end{equation}%
whose the trivial solution and non-trivial solution%
\begin{equation*}
a(s)=0
\end{equation*}%
and%
\begin{equation*}
a(s)=\frac{1}{s+c},
\end{equation*}%
respectively. We assume that $a(s)=0.$ By $(15)$ $b=b_{0}$ is a constant and
so is $c$. In that case by using $(9),(10)$ and $(11)$, $x$ and $y$ satisfy
the following differential equations
\begin{equation}
x^{2}\left( s\right) +y^{2}(s)=\lambda ^{2}\text{ \  \ }\lambda \text{ is a
non-zero constant,}
\end{equation}%
\begin{equation}
x(s)y^{\prime }(s)-x^{\prime }(s)y(s)=b_{0}\lambda ^{2},
\end{equation}%
\begin{equation}
x^{\prime }(s)y^{\prime \prime }-x^{\prime \prime }y^{\prime }(s)=b_{0}.
\end{equation}%
From $(18)$ we may put
\begin{equation}
x\left( s\right) =\lambda \cos \theta \left( s\right) ,\text{ \  \ }y\left(
s\right) =\lambda \sin \theta \left( s\right) ,
\end{equation}%
where $\theta \left( s\right) $ is some angle function. Differentiating $(21)
$ with respect to $s,$ we have
\begin{equation}
x^{\prime }(s)=-\theta ^{\prime }(s)y\left( s\right) \text{ \ and \ }%
y^{\prime }(s)=\theta ^{\prime }(s)x\left( s\right) .
\end{equation}%
By substituting $(21)$ and $(22)$ into $(19)$, we get
\begin{equation*}
\theta \left( s\right) =b_{0}s+d\text{, \  \ }d=const.
\end{equation*}%
And since the curve $\alpha $ is a unit speed curve, we have
\begin{equation*}
b_{0}^{2}\lambda ^{2}=1.
\end{equation*}%
Then we can write components of the curve $\alpha $ as:
\begin{equation*}
x\left( s\right) =\lambda \cos \left( b_{0}s+d\right) \text{ \  \ and \  \ }%
y\left( s\right) =\lambda \sin \left( b_{0}s+d\right) ,\text{ \  \  \ }%
b_{0}^{2}\lambda ^{2}=1.
\end{equation*}%
On the other hand, by using $(16)$ we can rewrite the Laplacian of the Gauss
map $G$ with $a(s)=0$ and $b=c=b_{0}$ as follows:%
\begin{equation*}
\Delta G=4b_{0}^{2}\left( e_{1}\wedge e_{2}\right)
\end{equation*}%
that is, the flat surface $M$ is pointwise 1-type Gauss map with the
function $f=4b_{0}^{2}$ and $C=0.$ Even if it is a pointwise 1-type Gauss
map of the first kind.

Now we assume that $a(s)=\frac{1}{s+c}.$ Since $b(s)$ is equal to $c(s),$
from $(15)$ we get
\begin{equation*}
b^{\prime }\left( s\right) =-a(s)b(s)
\end{equation*}%
or we can write%
\begin{equation*}
b^{\prime }\left( s\right) =-\frac{b(s)}{s+c},
\end{equation*}%
whose the solution
\begin{equation*}
b(s)=\mu a(s),\text{ \  \  \ }\mu \text{ is a constant.}
\end{equation*}%
By using $(16)$ we can rewrite the Laplacian of the Gauss map $G$ with $%
c(s)=b(s)=\mu a(s)$ as:%
\begin{equation}
\Delta G=\left( 4\mu ^{2}a^{2}\left( s\right) \right) \left( e_{1}\wedge
e_{2}\right) +2\mu a^{2}(s)\left( e_{1}\wedge e_{3}\right) -2\mu
a^{2}(s)\left( e_{2}\wedge e_{4}\right) .
\end{equation}%
We suppose that the flat rotational surface has pointwise 1-type Gauss map.
From $(1)$ and $(22)$, we get%
\begin{equation}
4\mu ^{2}a^{2}\left( s\right) =f+f\left \langle C,e_{1}\wedge
e_{2}\right \rangle
\end{equation}%
\begin{equation}
2\mu a^{2}(s)=f\left \langle C,e_{1}\wedge e_{3}\right \rangle
\end{equation}%
\begin{equation}
-2\mu a^{2}(s)=f\left \langle C,e_{2}\wedge e_{4}\right \rangle
\end{equation}%
Then, we have%
\begin{equation}
\left \langle C,e_{1}\wedge e_{4}\right \rangle =0,\text{ }\left \langle
C,e_{2}\wedge e_{3}\right \rangle =0,\text{ }\left \langle C,e_{3}\wedge
e_{4}\right \rangle =0
\end{equation}%
By using $(25)$ and $(26)$ we obtain
\begin{equation}
\left \langle C,e_{1}\wedge e_{3}\right \rangle +\left \langle C,e_{2}\wedge
e_{4}\right \rangle =0
\end{equation}%
By differentiating the first equation in $(27)$ with respect to $e_{1}$ and
by using $(8)$, the third equation in $(27)$ and $(28)$, we get
\begin{equation}
2a(s)\left \langle C,e_{1}\wedge e_{3}\right \rangle +\mu a(s)\left \langle
C,e_{1}\wedge e_{2}\right \rangle =0
\end{equation}%
Combining $(24),(25)$ and $(29)$ we then have%
\begin{equation}
f=4\left( a^{2}\left( s\right) +\mu ^{2}a^{2}\left( s\right) \right)
\end{equation}%
that is, a smooth function $f$ depends only on $s.$ By differentiating $f$
with respect to $s$ and by using the equality $a^{\prime }\left( s\right)
=-a^{2}\left( s\right) $, we get%
\begin{equation}
f^{\prime }=-2a(s)f
\end{equation}%
By differentiating $(25)$ with respect to s and by using $(8),(24)$, the
third equation in $(27),(30),(31)$ and the equality $a^{\prime }\left(
s\right) =-a^{2}\left( s\right) $, we have%
\begin{equation*}
\mu a^{3}=0
\end{equation*}%
Since $a(s)\neq 0$, it follows that $\mu =0.$ Then we obtain that $b=c=0.$
Then the surface $M$ is a part of plane.

Thus we can give the following theorem and corollary.

\begin{theorem}
\label{teo 1}Let $M$\ be the flat rotation surface given by the
parametrization (7). Then $M$ has pointwise 1-type Gauss map if and only if $%
M$ is either totally geodesic or parametrized by
\begin{equation}
X\left( s,t\right) =\left(
\begin{array}{c}
\lambda \cos \left( b_{0}s+d\right) \cos t,\lambda \cos \left(
b_{0}s+d\right) \sin t, \\
\lambda \sin \left( b_{0}s+d\right) \cos t,\lambda \sin \left(
b_{0}s+d\right) \sin t%
\end{array}%
\right) \text{,\  \  \ }b_{0}^{2}\lambda ^{2}=1
\end{equation}%
where $b_{0},$ $\lambda $ and $d$ are real constants.
\end{theorem}

\begin{corollary}
\label{cor 1}Let $M$\ be flat rotation surface given by the parametrization
(7). If $M$ has pointwise 1-type Gauss map then the Gauss map $G$ on $M$ is
of 1-type.
\end{corollary}

\section{ The general rotation surface and  Lie group }

In this section, we determine the profile curve of the general rotation
surface which has a group structure with the bicomplex number product.

Let the hyperquadric $P$ be given by%
\begin{equation*}
P=\left \{ x=\left( x_{1},x_{2},x_{3},x_{4}\right) \neq 0\text{; \  \  \ }%
x_{1}x_{4}=x_{2}x_{3}\right \} .
\end{equation*}%
We consider $P$ as the set of bicomplex number%
\begin{equation*}
P=\left \{ x=x_{1}1+x_{2}i+x_{3}j+x_{4}ij\text{ };\text{\ }%
x_{1}x_{4}=x_{2}x_{3},\text{ }x\neq 0\right \} .
\end{equation*}%
The components of $P$ are easily obtained by representing bicomplex number
multiplication in matrix form.%
\begin{equation*}
\tilde{P}=\left \{ M_{x}=\left(
\begin{array}{cccc}
x_{1} & -x_{2} & -x_{3} & x_{4} \\
x_{2} & x_{1} & -x_{4} & -x_{3} \\
x_{3} & -x_{4} & x_{1} & -x_{2} \\
x_{4} & x_{3} & x_{2} & x_{1}%
\end{array}%
\right) ;\text{\ }x_{1}x_{4}=x_{2}x_{3},\text{\ }x\neq 0\right \} .
\end{equation*}

\begin{theorem}
\label{teo 2}The set of $P$ together with the bicomplex number product is a
Lie group
\end{theorem}

\begin{proof}
$\tilde{P}$ is a differentiable manifold and at the same time a group with
group operation given by matrix multiplication. The group function%
\begin{equation*}
.:\tilde{P}\times \tilde{P}\rightarrow \tilde{P}
\end{equation*}%
defined by $\left( x,y\right) \rightarrow x.y$ is differentiable. So $(P,.)$
can be made a Lie group so that $g$ is a isomorphism \cite{yay}.
\end{proof}

\begin{remark}
\label{remark 2}The surface $M$ given by the parametrization (7) is a subset
of $P.$
\end{remark}

\begin{proposition}
\label{pro 1}Let $M$\ be a rotation surface given by the parametrization (7).
If $x(s)$ and $y(s)$ satisfy the following equations then $M$\ is a Lie
subgroup of $P$.%
\begin{equation}
x\left( s_{1}\right) x\left( s_{2}\right) -y\left( s_{1}\right) y\left(
s_{2}\right) =x\left( s_{1}+s_{2}\right)
\end{equation}%
\begin{equation}
x\left( s_{1}\right) y\left( s_{2}\right) +x\left( s_{2}\right) y\left(
s_{1}\right) =y\left( s_{1}+s_{2}\right)
\end{equation}%
\begin{equation}
\frac{x\left( s\right) }{x^{2}\left( s\right) +y^{2}\left( s\right) }%
=x\left( -s\right)
\end{equation}%
\begin{equation}
-\frac{y\left( s\right) }{x^{2}\left( s\right) +y^{2}\left( s\right) }%
=y\left( -s\right)
\end{equation}
\end{proposition}

\begin{proof}
Let $\alpha (s)=\left( x(s),0,y(s),0\right) $ be a profile curve of the
rotation surface given by the parametrization (7) such that $x(s)$ and $y(s)$
satisfy the equations $\left( 33\right) ,$ $\left( 34\right) ,$ $\left(
35\right) $ and $\left( 36\right) $. In that case we obtain that the inverse
of $X\left( s,t\right) $ is $X\left( -s,-t\right) $ and $X\left(
s_{1},t_{1}\right) \times X\left( s_{2},t_{2}\right) =X\left(
s_{1}+s_{2},t_{1}+t_{2}\right) .$ This completes the proof.
\end{proof}

\begin{proposition}
\label{pro 2}Let $\alpha (s)=\left( x(s),0,y(s),0\right) $ be a profile curve
of the rotation surface given by the parametrization (7) such that $x(s)$
and $y(s)$ satisfy the equation $x^{2}\left( s\right) +y^{2}\left( s\right)
=\lambda ^{2},$ where $\lambda $ is a non-zero constant. If $M$\ is a
subgroup of $P$ then the profile curve $\alpha $ is a unit circle.
\end{proposition}

\begin{proof}
We assume that $x(s)$ and $y(s)$ satisfy the equation $x^{2}\left( s\right)
+y^{2}\left( s\right) =\lambda ^{2}.$ Then we can put%
\begin{equation}
x(s)=\lambda \cos \theta \left( s\right) \text{ and }y(s)=\lambda \sin
\theta \left( s\right)
\end{equation}%
where $\lambda $ is a real constant and $\theta \left( s\right) $ is a
smooth function. Since $M$\ is a group, there exists one and only inverse of
all elements on $M.$ In that case the inverse of $X\left( s,t\right) $ is
given by%
\begin{equation*}
X^{-1}\left( s,t\right) =\left(
\begin{array}{c}
\frac{x\left( s\right) }{x^{2}\left( s\right) +y^{2}\left( s\right) }\cos
\left( -t\right) ,\frac{x\left( s\right) }{x^{2}\left( s\right) +y^{2}\left(
s\right) }\sin \left( -t\right) , \\
-\frac{y\left( s\right) }{x^{2}\left( s\right) +y^{2}\left( s\right) }\cos
\left( -t\right) ,-\frac{y\left( s\right) }{x^{2}\left( s\right)
+y^{2}\left( s\right) }\sin \left( -t\right)%
\end{array}%
\right)
\end{equation*}%
where
\begin{eqnarray}
\frac{x\left( s\right) }{x^{2}\left( s\right) +y^{2}\left( s\right) }
&=&x\left( f\left( s\right) \right) ,\text{ }  \notag \\
-\frac{y\left( s\right) }{x^{2}\left( s\right) +y^{2}\left( s\right) }
&=&y\left( f\left( s\right) \right) \text{, }f\text{ is a smooth function.}
\end{eqnarray}%
By using $(38)$, we get
\begin{equation}
x\left( s\right) =\lambda ^{2}x\left( f\left( s\right) \right)
\end{equation}%
and%
\begin{equation}
y\left( s\right) =-\lambda ^{2}y\left( f\left( s\right) \right)
\end{equation}%
By summing of the squares on both sides in $(39)$ and $(40)$ and by using $%
(37)$, we obtain that $\lambda ^{2}=1.$ This completes the proof.
\end{proof}

\begin{proposition}
\label{pro 3}Let $\alpha (s)=\left( x(s),0,y(s),0\right) $ be a profile curve
of the rotation surface given by the parametrization (7) such that $x(s)$
and $y(s)$ is given by $x(s)=\lambda \cos \theta \left( s\right) $ and $%
y(s)=\lambda \sin \theta \left( s\right) .$ Then if $\lambda =1$ and $\theta
$ is a linear function then $M$ is a Lie subgroup of $P$.
\end{proposition}

\begin{proof}
We assume that $\lambda =1$ and $\theta $ is a linear function. Then we can
write
\begin{equation*}
x(s)=\cos \eta s\text{ and }y(s)=\sin \eta s
\end{equation*}%
and in that case $x(s)$ and $y(s)$ satisfy the equations $\left( 33\right) ,$
$\left( 34\right) ,$ $\left( 35\right) $ and $\left( 36\right) .$ Thus from Proposition ($\ref{pro 2}$)
$M$ is a subgroup of $P.$ Also, it is a submanifold of $P.$
\end{proof}

\begin{proposition}
\label{pro 4}Let $\alpha (s)=\left( x(s),0,y(s),0\right) $ be a profile curve
of the rotation surface given by the parametrization (7) such that $x(s)$
and $y(s)$ is given by $x(s)=u(s)\cos \theta \left( s\right) $ and $%
y(s)=u(s)\sin \theta \left( s\right) .$ Then if $u:$ $\left( \mathbb{R}%
,+\right) \rightarrow \left( \mathbb{R}^{+},.\right) $ is a group
homomorphism and $\theta $ is a linear function then $M$ is a Lie subgroup
of $P$.
\end{proposition}

\begin{proof}
Let $x(s)$ and $y(s)$ be given by $x(s)=u(s)\cos \theta \left( s\right) $
and $y(s)=u(s)\sin \theta \left( s\right) $ and let $u:$ $\left( \mathbb{R}%
,+\right) \rightarrow \left( \mathbb{R}^{+},.\right) $ be a group
homomorphism and $\theta $ be a linear function. In that case $x(s)$ and $%
y(s)$ satisfy the equations $\left( 33\right) ,$ $\left( 34\right) ,$ $%
\left( 35\right) $ and $\left( 36\right) .$ Thus from Proposition ($\ref{pro 1}$) $M$ is
a subgroup of $P.$ Also, it is a submanifold of $P.$ So it is a Lie subgroup
of $P.$
\end{proof}

\begin{corollary}
\label{cor 2}Let $\alpha (s)=\left( x(s),0,y(s),0\right) $ be a profile curve
of the rotation surface given by the parametrization (7) such that $x(s)$
and $y(s)$ is given by $x(s)=\lambda \cos \theta \left( s\right) $ and $%
y(s)=\lambda \sin \theta \left( s\right) $ for $\theta $ linear function. If
$M$ is a Lie subgroup then $\lambda =1$.
\end{corollary}

\begin{proof}
We assume that $M$ is a group and $\lambda \neq 1.$ From Proposition ($\ref{pro 1}$) we
obtain that $\lambda =-1.$ On the other hand, for $\lambda =-1$ and $\theta $
linear function the closure property is not satisfied on $M.$ This is a
contradiction. Then we obtain that $\lambda =1$.
\end{proof}

\begin{remark}
\label{remark 3}Let $M$ be a Vranceanu surface. If the surface $M$ is flat
then it is given by
\begin{equation*}
X\left( s,t\right) =\left( e^{ks}\cos s\cos t,e^{ks}\cos s\sin t,e^{ks}\sin
s\cos t,e^{ks}\sin s\sin t\right)
\end{equation*}%
where $k$ is a real constant. In that case we can say that flat Vranceanu
surface is a Lie subgroup of $P$ with bicomplex multiplication. Also, flat
Vranceanu surface with pointwise 1-type Gauss map is Clifford torus and it
is given by%
\begin{equation*}
X\left( s,t\right) =\left( \cos s\cos t,\cos s\sin t,\sin s\cos t\sin s\sin
t\right)
\end{equation*}%
and Clifford Torus is a Lie subgroup of $P$ with bicomplex multiplication.
See for more details \cite{ak}.
\end{remark}

\begin{theorem}
\label{teo 3}Let $M$\ be non-planar flat rotation surface with pointwise
1-type Gauss map given by the parametrization (32) with $d=2k\pi $. Then $M$
is a Lie group with bicomplex multiplication if and only if it is a Clifford
torus.
\end{theorem}

\begin{proof}
We assume that $M$ is a Lie group with bicomplex multiplication then from Corollary ($\ref{cor 2}$)
we get that $\lambda =1.$ Since $b_{0}^{2}\lambda ^{2}=1,$ it
follows that $b_{0}=\varepsilon $, where $\varepsilon =\pm 1.$ In that case
the surface $M$ is given by
\begin{equation*}
X\left( s,t\right) =\left( \cos \varepsilon s\cos t,\cos \varepsilon s\sin
t,\sin \varepsilon s\cos t\sin \varepsilon s\sin t\right)
\end{equation*}%
and $M$ is a Clifford torus, that is, the product of two plane circle wih
the same radius.

Conversely, Clifford torus is a flat rotational surface with pointwise
1-type Gauss map the surface which can be obtained by the parametrization $%
(32)$ and it is a Lie group with bicomplex multiplication. This completes
the proof.
\end{proof}

\end{document}